\documentclass[10pt,twoside,a4paper,english,reqno]{amsart}
\usepackage[T1]{fontenc}
\usepackage{lmodern}
\usepackage{amscd}
\usepackage{caption}
\usepackage{a4wide}
\usepackage[protrusion=true,expansion=false]{microtype}
\usepackage{babel}
\usepackage{doi}
\usepackage{amssymb,amsthm,amsmath,amsfonts,mathrsfs}
\usepackage{bm,latexsym,enumitem,dsfont,tabularx,graphicx,url}
\usepackage[utf8]{inputenc}
\usepackage{hyperref}
\usepackage{stmaryrd}
\hypersetup{
    colorlinks=true,
    linkcolor=blue,
    filecolor=red,
    urlcolor=blue,
    citecolor=red,
}
\usepackage{nicefrac,upref,ulem,doi,footmisc}
\usepackage[svgnames]{xcolor}
\usepackage{orcidlink}
\theoremstyle{plain}
\newtheorem{theorem}{Theorem}[section]

\newtheorem{lemma}[theorem]{Lemma}

\newtheorem{remark}[theorem]{Remark}

\newtheorem{corollary}[theorem]{Corollary}
\numberwithin{equation}{section}

\usepackage{tikz,tikzscale}
\usepackage{pgfplots}
\usepackage{subcaption}
\pgfplotsset{compat=1.13}

\definecolor{lateksii_color}{RGB}{155,0,119}

\title[HDG for Einstein scalar equation]{A Convergent Hybridizable Discontinuous Galerkin Method for Einstein--Scalar Equations}

\author{Mukul Dwivedi\,\orcidlink{0000-0002-5891-2788} \and Andreas Rupp\,\orcidlink{0000-0001-5527-7187}}
\address{Department of Mathematics, Saarland University, Saarbrücken, Germany}
\email{\{mukul.dwivedi;andreas.rupp\}@uni-saarland.de}

\thanks{This work has been supported by the Deutsche Forschungsgemeinschaft (DFG, German
Research Foundation) -- 577175348.}

\newcommand{\Th}{\mathcal T_h}

\begin{document}

\begin{abstract}
We propose and analyze a hybridized discontinuous Galerkin (HDG) method for the spherically symmetric Einstein--scalar system in Bondi gauge. After rewriting the model as a local first-order PDE--ODE system by introducing suitable scaled variables, we construct a semidiscrete scheme in which the element unknowns are computed locally and the coupling is carried by traces on the mesh skeleton. In the present radial setting, these traces can be eliminated recursively, so that only the main evolution variable is
advanced in time, while the metric variables are recovered from discrete constraint relations. We prove local semidiscrete well-posedness, derive a global \(L^2\)--stability estimate, establish an optimal order \(L^2\) error bound for the main evolution variable for polynomial degree \(k\ge 1\), and obtain reconstruction error estimates for the metric variables and the associated mass functional. Numerical experiments verify the predicted spatial convergence rate and illustrate qualitative features of the Einstein--scalar dynamics, including large-data collapse profiles and smooth-pulse evolution.
\end{abstract}

\maketitle

\section{Introduction}

Numerical solutions of the Einstein equations play a fundamental role in modern gravitational physics, enabling the simulation of black holes, neutron stars, and gravitational wave sources; see, for instance, \cite{robert1997einstein,Inverno2000Cauchy,omez1986numerical,Omez1998moving}. The field of numerical relativity traces its origins to the earliest computational formulations of Einstein's theory and the first numerical explorations of black-hole spacetimes \cite{arnowitt1962dynamics,hahn1964two}. Following a prolonged period of challenges concerning coordinate choices, boundary conditions, and numerical stability, significant breakthroughs were achieved through the development of robust formulations and suitable gauge conditions. These advances ultimately enabled stable and accurate simulations of binary black-hole mergers; see \cite{baker2006gravitational,campanelli2006accurate,pretorius2005evolution,pretorius2005numerical} and the references therein. For the detailed background on numerical relativity and the Bondi-gauge formulation considered here, we refer to the references above and to \cite{chen2026convergence}.

From a practical standpoint, numerical relativity has achieved remarkable success. Nonetheless, the stability and convergence analysis of high-order methods for the full nonlinear Einstein system
remains mathematically difficult. Under suitable decompositions and gauge choices, Einstein's equations become highly nonlinear hyperbolic systems, and rigorous numerical analysis in this
setting is still rather limited. It is therefore natural to begin with symmetry-reduced models. Among these, the spherically symmetric Einstein--scalar field system is particularly important, since it retains key analytical and physical features such as gravitational collapse, horizon formation, and long-time metric evolution. The foundational PDE theory for this model was developed by Christodoulou; see
\cite{christodoulou1986global,christodoulou1986problem, christodoulou1987mathematical,christodoulou1987structure,christodoulou1991formation}. More recently, Chen et al. \cite{chen2026convergence} proved $L^2$-stability and optimal-order convergence of a discontinuous Galerkin method for the Bondi-gauge formulation, including large-data regimes that lead to collapse toward black-hole-type profiles.

We briefly recall the model following the derivation in Bondi gauge in \cite{chen2026convergence} and the classical works on the spherically symmetric Einstein--massless-scalar system \cite{christodoulou1986global}. In geometric units ($G=c=1$), the Einstein field equations are
\begin{equation}\label{eq:Einstein}
G_{\mu\nu} = 8\pi\,T_{\mu\nu},
\end{equation}
where $G_{\mu\nu}:=R_{\mu\nu}-\tfrac12 g_{\mu\nu}R$ is the Einstein tensor. Here, $g_{\mu\nu}$ is the spacetime metric, $R_{\mu\nu}$ is the Ricci tensor, and $R$ is the scalar curvature.
For a massless scalar field $\phi$, the stress--energy tensor is
\[
T_{\mu\nu} = \partial_\mu\phi\,\partial_\nu\phi
-\frac12\,g_{\mu\nu}\,g^{\alpha\beta}\,\partial_\alpha\phi\,\partial_\beta\phi.
\]
In spherical symmetry, adopting Bondi-type null coordinates $(t,r,\theta,\varphi)$ with areal radius $r\ge 0$,
one may write the metric in the Bondi--Sachs form \cite{bondi1962gravitational,chen2026convergence,sachs1962gravitational}
\begin{equation}\label{eq:BondiMetric}
ds^2 = -g(t,r)\,\tilde g(t,r)\,dt^2 - 2\,g(t,r)\,dt\,dr + r^2\,d\Omega^2,
\qquad d\Omega^2 = d\theta^2 + \sin^2\theta\,d\varphi^2,
\end{equation}
where $g$ and $\tilde g$ are positive metric coefficients.

The reduced equations in \eqref{eq:BondiMetric} are obtained by combining two ingredients.
First, suitable components of the Einstein equations provide radial ``constraint'' relations for the metric
coefficients. Second, the conservation law $\nabla_\mu T^{\mu\nu}=0$ yields the wave equation for $\phi$,
which couples back to the metric through $g$ and $\tilde g$; see \cite[\S2 and Appendix~A]{chen2026convergence}
for a detailed derivation in the present normalization.
To express the resulting system in a simple first-order form, one introduces the auxiliary variables
\begin{subequations}\label{eq:ES_aux}
\begin{align}
u(t,r) &:=\bigl(r\,\phi(t,r)\bigr)_r, \label{eq:ES_uphi}\\
\tilde u(t,r) &:= \phi(t,r) = \frac1r\int_0^r u(t,s)\,ds. \label{eq:ES_tildeu}
\end{align}
\end{subequations}
Here, the averaging relation \eqref{eq:ES_tildeu} encodes regularity at $r=0$, and it is already the first place
where a nonlocal dependence on the full interval $[0,r]$ enters the formulation.

With this notation (and using the same normalization as \cite{chen2026convergence}, in which the factor $4\pi$
is absorbed), the symmetry-reduced Einstein--scalar system becomes the coupled PDE--ODE system \cite{chen2026convergence}:
\begin{subequations}\label{eq:ES_original}
\begin{align}
u_t - \!\Big(\frac12\,\tilde g\,u\Big)_r &= -\frac12\,\tilde g_r\,\tilde u, \label{eq:ES_u}\\
g_r &= \frac1r\,g\,\bigl(u-\tilde u\bigr)^2, \label{eq:ES_g}\\
\tilde g_r &= \frac1r\,(g-\tilde g). \label{eq:ES_gtilde} 
\end{align}
\end{subequations}
We consider $r\in[0,b]$ for a fixed outer radius $b>0$ and $(t,r)\in[0,T]\times[0,b]$.
The outer boundary data are prescribed at $r=b$ by
\begin{equation}\label{eq:outerBC}
u(t,b)=U_b(t), \qquad g(t,b)=1,
\end{equation}
while regularity at the center $r=0$ requires $r\phi(t,r)\to 0$ as $r\to 0$ and $\tilde g$ remains bounded. A useful point for both analysis and numerics is that the constraint equations \eqref{eq:ES_g} and \eqref{eq:ES_gtilde}
can be integrated explicitly (for each fixed $t$), which makes the nonlocal structure completely visible:
\begin{subequations}\begin{align}\label{ES_gatra}
    \tilde g(r) &= \frac{1}{r}\int_0^r g(s)\,ds,\\  \label{ES_gatrb}
    g(r) &= \exp \Big(-\int_r^b \frac{1}{s}(u -\tilde u)^2\,ds\Big).
\end{align}\end{subequations}
Thus, even though \eqref{eq:ES_u} is written as a local transport equation, its coefficients and source term depend
on spatial averages and integrals through \eqref{eq:ES_tildeu}, \eqref{ES_gatra}, and \eqref{ES_gatrb}.

A key structural feature is that \eqref{eq:ES_tildeu} and the representation of $\tilde g$ are nonlocal in $r$.
To localize these equations, we introduce \( w(t,r) := r\,\tilde u(t,r) \) and  \( z(t,r):= r\,\tilde g(t,r)\).
Then, from \eqref{eq:ES_tildeu},
\[
w_r=(r\tilde u)_r=u.
\]
Likewise, using \eqref{ES_gatra},
\[
z_r=(r\tilde g)_r=g.
\]
The corresponding center conditions are
\begin{equation}\label{eq:centerBC}
w(t,0)=0, \qquad z(t,0)=0.
\end{equation}
Thus, the model \eqref{eq:ES_tildeu}--\eqref{eq:ES_original} is equivalently represented by the local first-order system
\begin{subequations}\label{eq:ES_local}
\begin{align}
u_t - \!\Big(\frac12\,\tilde g\,u\Big)_r &= -\frac12\,\tilde g_r\,\tilde u, \label{eq:local_u}\\
w_r &= u, \label{eq:local_w}\\
g_r &= \frac1r\,g\,\bigl(u-\tilde u\bigr)^2, \label{eq:local_g}\\
z_r &= g, \label{eq:local_z}
\end{align}
\end{subequations}
with $\tilde u=w/r$, $\tilde g=z/r$ for $r>0$, center conditions \eqref{eq:centerBC}, and outer conditions \eqref{eq:outerBC}.
This local form is the starting point for our HDG formulation.

We also record a basic bound for the metric coefficients, which will be repeatedly used to identify the correct upwind directions and to control the coefficients in the numerical fluxes.

\begin{lemma}[Bounds for the metric coefficients]\label{lem:metric_bounds}
Assume $g$ and $\tilde g$ satisfy \eqref{eq:ES_g}--\eqref{eq:ES_gtilde} with the outer normalization $g(b)=1$.
Then for every fixed $t$ and all $r\in(0,b]$,
\[
0 < \tilde g(r) \le g(r) \le 1.
\]
\end{lemma}

\begin{proof}
 From \eqref{eq:ES_g}, we have $g_r = \frac1r g (u-\tilde u)^2 \ge 0$ for $r>0$. Hence, $g$ is nondecreasing in $r$ and $g(r)\le g(b)=1$. Moreover, \eqref{ES_gatra} implies $\tilde g(r) \le g(r)$, and \eqref{ES_gatrb} implies $g> 0$. This immediately yields $0< \tilde g\le 1$.
\end{proof}

These considerations suggest that the spherically symmetric Einstein--scalar system \eqref{eq:ES_original} provides a natural model for testing new high-order discontinuous discretizations. To obtain a numerical method with a better computational structure, high-order accuracy, and fewer globally coupled degrees of freedom, we consider a hybridizable discontinuous Galerkin (HDG) approach. HDG methods were first developed for second-order elliptic problems \cite{cockburn2009unified}, and have since been successfully extended to many other classes of partial differential equations. 
These include convection--diffusion problems \cite{cockburn2009hdgcondiff,Ngunyen2009hdgimplicitnonlinear}, Stokes and incompressible flow systems \cite{cockburn2009hdgstokesderi,cockburn2009hdgstokesanaly,Lu2024hdghomogenstokes,Ngunyen2009hdgstokes,Nguyen2011incompre,Rhebergen2012incompre}, the Helmholtz equation \cite{Huangxin2013Helmhol,Griesmair2011Helmhotz,Jintao2014Helmhol}, and elasticity equations \cite{Cockburn2013hdglinearelast,Kabaria2013hdgnonlinearelast,Qiu2018hdglinearelast}. 
More recently, HDG techniques have also been developed for nonlinear dispersive and fully nonlinear problems, including the Ostrovsky equation \cite{dwivedi2026hybridizable}, the Camassa--Holm--Kadomtsev--Petviashvili equation \cite{dwivedi2026hybridizablechkp}, the Monge--Amp\`ere problem \cite{Nguyen2024hdgMonge}, and KdV-type equations \cite{Chen2018KdVfifth,dong2017optimally,Samii2016KdV}.
A defining feature of HDG is that the globally coupled unknowns are numerical traces on the mesh skeleton, while the element unknowns are recovered locally. This structure not only reduces the size of the global algebraic system compared with standard DG methods, but also makes HDG particularly attractive from the solver point of view. In this direction, Lu et al. developed homogeneous multigrid methods for HDG discretizations and provided an analysis of injection operators for geometric multigrid solvers for HDG methods \cite{Lu2022Analysis,Lu2022Homomulti}. Related multiscale and localized formulations have also been investigated in the HDG setting; see, for instance, \cite{LuMaierRupp2025}.
For the Bondi-gauge Einstein--scalar system considered here, such a structure is particularly attractive because the evolution equation is coupled with radial constraint relations, and therefore a method based on local reconstruction and reduced global coupling provides a natural and efficient discretization framework.

The main contribution of this paper is to formulate and analyze a hybridized trace-based semidiscretization for the Bondi--gauge Einstein--scalar system \eqref{eq:ES_original} on a finite radial interval. After introducing the
auxiliary variables \(w=r\tilde u\) and \(z=r\tilde g\), the nonlocal structure
of the model is rewritten in the local first-order form \eqref{eq:ES_local},
while the discrete reconstruction is carried out through the local reconstruction formulae, which makes it possible to couple the elements only through single-valued traces. In the present one-dimensional radial setting, these traces can in fact be eliminated recursively, so that only the principal evolution variable \(u_h\) is advanced in time, while the remaining quantities are reconstructed from the discrete constraint relations.
On the analytical side, we prove local semidiscrete well--posedness, derive a global \(L^2\)--stability estimate, establish an optimal order \(L^2\) error bound for \(u_h\) for polynomial degree \(k\ge 1\), and obtain corresponding reconstruction estimates for
the metric variables and the associated mass functional.
The numerical experiments support the theory and demonstrate that the method reproduces the expected qualitative behavior of the Einstein--scalar dynamics, including collapse toward black-hole profiles. In this sense, the paper shows that HDG provides not only a reduced discrete coupling structure but also a rigorous and effective framework for high-order approximation of this symmetry-reduced Einstein system.

The paper is organized as follows. In Section \ref{sec:hdg}, we introduce the HDG discretization and derive the local and global trace formulations. Section \ref{sec:hdg-stability} is devoted to the semidiscrete well-posedness and the $L^2$-stability estimate. In Section \ref{sec:hdg-error}, we prove the optimal error estimate and the corresponding reconstruction-error bounds. Section \ref{sec:numerical} presents numerical experiments that verify the theoretical rates and illustrate black-hole formation and the associated metric behavior. We close with a brief discussion in Section \ref{sec:conclusion}.

\section{A hybridizable discontinuous Galerkin discretization}\label{sec:hdg}

We discretize the local first--order Einstein--scalar system \eqref{eq:ES_local} in space on $[0,b]$ and keep time continuous.

\subsection{Mesh, skeleton notation, and discrete spaces}\label{subsec:spaces}

Let $0=r_0<r_1<\dots<r_N=b$ be a partition of $[0,b]$ and set
\[
I_i:=(r_{i-1},r_i),\qquad i=1,\dots,N,\qquad h_i = r_i-r_{i-1},\qquad \mathcal T_h:=\{I_i\}_{i=1}^N,\qquad h= \max h_i.
\]
We denote the set of mesh nodes by
\[
\mathcal F_h:=\{r_i\}_{i=0}^N,\qquad \mathcal F_h^\partial:=\{r_0,r_N\},\qquad
\mathcal F_h^0:=\mathcal F_h\setminus \mathcal F_h^\partial.
\]
For an integer $k\ge 0$, let $ P^k(I_i)$ be the polynomials of degree at most $k$ on $I_i$, and
define the broken space
\[
V_h^k:=\bigl\{v\in L^2(0,b):\ v|_{I_i}\in  P^k(I_i)\ \text{for }i=1,\dots,N\bigr\}.
\]
We use the standard volume inner product
\[
(v_1,v_2)_{I_i} := \int_{I_i} v_1 v_2\,dr,
\qquad
(v_1,v_2)_{\mathcal T_h}:=\sum_{i=1}^N (v_1,v_2)_{I_i},
\]
and the oriented boundary pairing
\[
\langle \psi, \eta\,n\rangle_{\partial I_i} := \psi(r_i)\,\eta(r_i) - \psi(r_{i-1})\,\eta(r_{i-1}),
\]
where the outward normal on $\partial I_i$ is $n(r_{i-1})=-1$ and $n(r_i)=+1$.
We also introduce the boundary--restricted trace spaces
\[
M_h^{R}:= L^2(\mathcal F_h\setminus \{r_0\}),\qquad
M_h^{L}:=L^2(\mathcal F_h\setminus \{r_N\}),\qquad
M_h^{0}:=L^2(\mathcal F_h^0).
\]

Fix an element $I_i$. Let us consider system \eqref{eq:ES_local} in $I_i$, and prescribe the following
\begin{equation}\label{eq:cont-inflow-data}
u(t,r_i)=\widehat u_i(t),\qquad g(t,r_i)=\widehat g_i(t),\qquad
w(t,r_{i-1})=\widehat w_{i-1}(t),\qquad z(t,r_{i-1})=\widehat z_{i-1}(t),
\end{equation}
along with the initial data $u_0|_{I_i}$.
Assume the unique solvability of the local system \eqref{eq:ES_local} on an element
$I_i=(r_{i-1},r_i)$ for given data \eqref{eq:cont-inflow-data} and a sufficiently smooth
initial condition $u_0|_{I_i}$, and denote the corresponding elementwise solution by
$(u,w,g,z)|_{I_i}$. Define $(u,w,g,z)$ on $[0,b]$ by collecting these element solutions.

This piecewise function $(u,w,g,z)$ is a global (classical) solution of the original problem on $[0,b]$
if and only if the traces of the quantities that are differentiated in $r$ are single-valued at
every interior node. Concretely, for each $r_i\in\mathcal F_h^0$ (with $i=1,\dots,N-1$) we require
\begin{equation}\label{eq:cont-transmission}
\begin{aligned}
w|_{I_i}(t,r_i)=w|_{I_{i+1}}(t,r_i) &:= \widehat w_i(t),\qquad 
z|_{I_i}(t,r_i)=z|_{I_{i+1}}(t,r_i) := \widehat z_i(t),\\
u|_{I_i}(t,r_i)=u|_{I_{i+1}}(t,r_i) &:= \widehat u_i(t), \qquad
g|_{I_i}(t,r_i)=g|_{I_{i+1}}(t,r_i) := \widehat g_i(t).
\end{aligned}
\end{equation}

The physical boundary data close the system by fixing the remaining inflow values at $r=0$ and $r=b$ by
\[
\widehat u_N(t)=U_b(t),\qquad \widehat g_N(t)=1,\qquad
\widehat w_0(t)=0,\qquad \widehat z_0(t)=0.
\]
Our HDG discretization is then designed as a discrete analogue of this characterization: element
unknowns are computed locally for given traces, and the traces are determined globally by enforcing a
weak form of \eqref{eq:cont-transmission} together with the boundary conditions.

\subsection{HDG method}\label{subsec:hdg-scheme}
For a fixed element \(I_i=(r_{i-1},r_i)\), assume that the traces
\(
\widehat u_h(r_i),~ \widehat g_h(r_i)\), \(\widehat w_h(r_{i-1}), ~\widehat z_h(r_{i-1})
\), and a function \(u_h|_{I_i}\in  P^k(I_i)\) are provided. We define the local reconstructions on $ I_i$ %\textcolor{red}{(MD: These reconstructed quantities are, in general, not polynomials on $I_i$)} 
as
\begin{subequations}\label{eq:hdg-local-reconstruction}
\begin{align}
w_h(r) &:= \widehat w_h(r_{i-1}) + \int_{r_{i-1}}^r u_h(s)\,ds,
\label{eq:hdg-local-reconstruction-w}\\
g_h(r) &:= \widehat g_h(r_i)\,
\exp\!\Bigl(-\int_r^{r_i}\frac{1}{s}\Bigl(u_h(s)-\tilde u_h(s)\Bigr)^2 ds\Bigr),
\label{eq:hdg-local-reconstruction-g}\\
z_h(r) &:= \widehat z_h(r_{i-1}) + \int_{r_{i-1}}^r g_h(s)\,ds,
\label{eq:hdg-local-reconstruction-z}
\end{align}
\end{subequations}
where, for \(r>0\),  \(
\tilde u_h(r):=\frac{w_h(r)}{r}\) with the endpoint value \( \tilde u_h(0):=u_h(0)\).
Thus, we seek the local unknown \(u_h\in  P^k(I_i)\) that  satisfies
\begin{equation}\label{eq:hdg-local-u}
\bigl((u_h)_t,\phi_u\bigr)_{I_i}
+\Bigl(\frac12\,\tilde g_h\,u_h,(\phi_u)_r\Bigr)_{I_i}
-\Bigl\langle \widehat{\mathcal F}_{u,h},\phi_u\,n\Bigr\rangle_{\partial I_i}
+\Bigl(\frac12\,(\tilde g_h)_r\,\tilde u_h,\phi_u\Bigr)_{I_i}
=0
\qquad\forall \phi_u\in P^k(I_i),
\end{equation}
where, for \(r>0\), \( \tilde g_h(r):=\frac{z_h(r)}{r}\) with the endpoint value \(\tilde g_h(0):=g_h(0) \), and the numerical flux on $I_i$ is
\begin{subequations}\label{eq:hdg-flux}
\begin{align}
\widehat{\mathcal F}_{u,h}(r_i) &:= \frac12\,\tilde g_h(r_i)\,\widehat u_h(r_i),
\label{eq:hdg-flux-right}\\
\widehat{\mathcal F}_{u,h}(r_{i-1}) &:= \frac12\,\tilde g_h(r_{i-1})\,u_h(r_{i-1}).
\label{eq:hdg-flux-left}
\end{align}
\end{subequations}
Consequently, we assemble criteria for the global trace unknowns $\widehat u_h,\,\widehat g_h\in M_h^{R},~
\widehat w_h,\,\widehat z_h\in M_h^{L}$ from which $u_h,w_h,g_h$ and $z_h$ can be derived in an element local procedure via nodal transmission conditions
\begin{subequations}\label{eq:hdg-transmission}
\begin{align}
\widehat{\mathcal F}_{u,h}|_{I_i}(r_i)
&=\widehat{\mathcal F}_{u,h}|_{I_{i+1}}(r_i),
\label{eq:hdg-transmission-u}\\
  w_h|_{I_i}(r_i)&= w_h|_{I_{i+1}}(r_i),
\label{eq:hdg-transmission-w}\\
z_h|_{I_i}(r_i)&=z_h|_{I_{i+1}}(r_i),
\label{eq:hdg-transmission-z}\\
g_h|_{I_i}(r_i)&=g_h|_{I_{i+1}}(r_i),
\label{eq:hdg-transmission-g}
\end{align}
\end{subequations}
for $i=1,\dots,N-1$. The boundary data are
\begin{equation}\label{eq:hdg-boundary-traces}
\widehat u_h(t,b)=U_b(t),\qquad
\widehat g_h(t,b)=1,\qquad
\widehat w_h(t,0)=0,\qquad
\widehat z_h(t,0)=0.
\end{equation}
\subsection{Local well--posedness} For a given \(u_h\in V_h^k\), the transmission conditions
\eqref{eq:hdg-transmission} and the boundary conditions
 \eqref{eq:hdg-boundary-traces} determine the traces uniquely. 
More precisely, the traces are given by the recursive formulas
\begin{subequations}\label{eq:trace-recursions}
\begin{align}
\widehat w_h(r_i) &=w_h|_{I_{i+1}}(r_i) = w_h|_{I_{i}}(r_i)
=
\widehat w_h(r_{i-1}) +
\int_{r_{i-1}}^{r_i} u_h(s)\,ds, 
\label{eq:trace-recursion-w}\\
\widehat g_h(r_i)
&= g_h|_{I_i}(r_i) =g_h|_{I_{i+1}}(r_i) =
\widehat g_h(r_{i+1})
\exp\!\Bigl(
-\int_{r_i}^{r_{i+1}}
\frac1s\bigl(u_h(s)-\tilde u_h(s)\bigr)^2\,ds
\Bigr), 
\label{eq:trace-recursion-g}\\
\widehat z_h(r_i) &= z_h|_{I_{i+1}}(r_i) = z_h|_{I_{i}}(r_i)
=\widehat z_h(r_{i-1}) +
\int_{r_{i-1}}^{r_i} g_h(s)\,ds,
\label{eq:trace-recursion-z}\\
\frac12\,\tilde g_h(r_i)\,\widehat u_h(r_i)&=\frac12\,\tilde g_h(r_i)\, u_h|_{I_{i+1}}(r_i)  \Longrightarrow \widehat u_h(r_i) = u_h|_{I_{i+1}}(r_i),
\label{eq:trace-recursion-u}
\end{align}
\end{subequations}
for $i=1,\dots,N-1$, where all chains are started by values in \eqref{eq:hdg-boundary-traces} and $\tilde g_h(r_i)>0$ from \eqref{eq:disc-metric-bounds}. Thus, in the present one-dimensional radial setting, the trace variables can be eliminated recursively once \(u_h\) is known, and the semidiscrete evolution
may be viewed as an ODE for the coefficients of $u_h$ alone.

\begin{theorem}[Local semidiscrete well-posedness]
\label{thm:local-wellposedness-hdg}
Let \(U_b\in C([0,T])\) and let \(u_h(0)\in V_h^k\) be given.
Then there exists \(T_\ast\in(0,T]\) such that the  scheme \eqref{eq:hdg-local-u}, \eqref{eq:hdg-transmission}, and
\eqref{eq:hdg-boundary-traces} admits a unique solution
\[
u_h\in C^1([0,T_\ast];V_h^k).
\]
The associated traces \( \widehat u_h,\; \widehat g_h,\; \widehat w_h,\; \widehat z_h \) are uniquely determined by \eqref{eq:trace-recursions}.
\end{theorem}

\begin{proof}
Choose on each element \(I_i\) a basis \(\{\varphi_{i,m}\}_{m=0}^k\) of \(P^k(I_i)\), and write
\[
u_h|_{I_i}(r,t)=\sum_{m=0}^k U_{i,m}(t)\,\varphi_{i,m}(r).
\]
Let \(U(t)\in \mathbb R^{N(k+1)}\) be the coefficient vector. For a given \(U\), the traces are uniquely determined by \eqref{eq:trace-recursions}. In particular,
\(\widehat w_h\) and \(w_h\) depend linearly on \(U\). On the first element \(I_1=(0,r_1)\),
\[
w_h(r)=\int_0^r u_h(s)\,ds,
\]
hence \(w_h(r)=r\,p_h(r)\) for a unique polynomial \(p_h\in P^k(I_1)\). Therefore \(
\tilde u_h(r)=p_h(r)\text{ on }I_1,
\) and \(
u_h(r)-\tilde u_h(r)=r\,q_h(r)
\) for some polynomial \(q_h\in P^{k-1}(I_1)\), as \( u_h(0) =\tilde u_h(0)\). It follows that
\(
\frac1r\bigl(u_h-\tilde u_h\bigr)^2
\)
extends continuously to \(r=0\) and depends polynomially on \(U\). On the remaining elements,
\(r\ge r_{i-1}>0\), so there is no singularity. Hence, by \eqref{eq:hdg-local-reconstruction-g} and \eqref{eq:trace-recursion-g}, the maps
\[
U\mapsto g_h,\qquad U\mapsto \widehat g_h
\]
are locally Lipschitz. The same is then true for
\(
U\mapsto z_h,\; U\mapsto \widehat z_h
\)
by \eqref{eq:hdg-local-reconstruction-z} and \eqref{eq:trace-recursion-z}. For \(i\ge2\),
\[
\tilde g_h=\frac{z_h}{r},
\qquad
(\tilde g_h)_r=\frac{g_h-\tilde g_h}{r},
\]
so \(U\mapsto \tilde g_h\) and \(U\mapsto (\tilde g_h)_r\) are locally Lipschitz on \(I_i\).
On \(I_1\),
\[
z_h(r)=\int_0^r g_h(s)\,ds,
\qquad
\tilde g_h(r)=\frac1r\int_0^r g_h(s)\,ds=\int_0^1 g_h(\theta r)\,d\theta.
\]
Moreover, differentiating \eqref{eq:hdg-local-reconstruction-g} gives
\[
(g_h)_r=\frac1r\,g_h\,(u_h-\tilde u_h)^2,
\]
so \((g_h)_r\) is continuous on \(I_1\), and therefore
\[
(\tilde g_h)_r(r)=\int_0^1 \theta\,(g_h)_r(\theta r)\,d\theta.
\]
Thus \(U\mapsto \tilde g_h\) and \(U\mapsto (\tilde g_h)_r\) are locally Lipschitz on \(I_1\) as well.
Finally, \(\widehat u_h(b)=U_b(t)\), while \(\widehat u_h(r_i)=u_h|_{I_{i+1}}(r_i)\) at the interior nodes;
hence the only explicit time dependence enters through \(U_b\), and the trace \(\widehat u_h\) is continuous
in \(t\) and locally Lipschitz in \(U\).

Testing \eqref{eq:hdg-local-u} with the basis functions \(\varphi_{i,m}\) yields a finite-dimensional system
\[
M_h\,\dot U(t)=\mathcal R(t,U(t)),
\]
where \(M_h\) is the block diagonal mass matrix. Each block is the mass matrix on \(P^k(I_i)\),
hence \(M_h\) is symmetric positive definite and therefore invertible. By the preceding argument,
\(\mathcal R\) is continuous in \(t\) and locally Lipschitz in \(U\). Therefore
\[
\dot U(t)=M_h^{-1}\mathcal R(t,U(t))
\]
is a finite-dimensional ODE with continuous time dependence and locally Lipschitz state dependence.
By the Picard--Lindel\"of theorem, there exists \(T_\ast\in(0,T]\) such that this ODE admits a unique solution
\[
U\in C^1([0,T_\ast];\mathbb R^{N(k+1)}).
\]
This defines a unique semidiscrete solution
\[
u_h\in C^1([0,T_\ast];V_h^k),
\]
and the associated traces are uniquely recovered from \eqref{eq:trace-recursions}.
\end{proof}

\begin{lemma}[Exact discrete constraint reconstruction]
\label{lem:hdg-metric-reconstruction}
Let \((u_h,w_h,g_h,z_h,\widehat u_h,\widehat w_h,\widehat g_h,\widehat z_h)\) be
a solution of
\eqref{eq:hdg-local-reconstruction}--\eqref{eq:hdg-boundary-traces}.
Then, for each fixed \(t\in[0,T]\),
\begin{align}
(w_h)_r &= u_h \qquad\text{on each } I_i,
\label{eq:disc-w-strong}\\
(g_h)_r &= \frac1r\,g_h\,(u_h-\tilde u_h)^2 \qquad\text{on each } I_i,
\label{eq:disc-g-strong}\\
(z_h)_r &= g_h \qquad\text{on each } I_i.
\label{eq:disc-z-strong}
\end{align}
Moreover, the transmission conditions \eqref{eq:hdg-transmission} imply the global representation formulas 
\begin{equation}\label{eq:disc-global-representations}
w_h(r)=\int_0^r u_h(s)\,ds,\qquad
g_h(r)=\exp\!\Bigl(-\int_r^b \frac1s\bigl(u_h(s)-\tilde u_h(s)\bigr)^2\,ds\Bigr),\qquad
z_h(r)=\int_0^r g_h(s)\,ds,
\end{equation}
for all \(r\in [0,b]\). Consequently,
\begin{equation}\label{eq:disc-metric-bounds}
0<g_h(r)\le1,\qquad
0< \tilde g_h(r)=\frac{z_h(r)}{r}\le g_h(r)\le1
\qquad\text{for all } r\in(0,b],
\end{equation}
and
\begin{equation}\label{eq:disc-tildeg-identity}
(\tilde g_h)_r = \frac1r\,(g_h-\tilde g_h)\ge0,
\qquad
(\tilde g_h)_r\le \frac{g_h}{r}
\qquad\text{on }(0,b].
\end{equation}
Finally, \( \tilde u_h(0)=u_h(0)\) and \(\tilde g_h(0)=g_h(0)\).
\end{lemma}

\begin{proof}
Differentiating \eqref{eq:hdg-local-reconstruction-w} and
\eqref{eq:hdg-local-reconstruction-z} gives
\eqref{eq:disc-w-strong} and \eqref{eq:disc-z-strong} respectively. Likewise,
differentiating \eqref{eq:hdg-local-reconstruction-g} yields
\eqref{eq:disc-g-strong}. Because \(\widehat w_h(0)=0\) and \(\widehat z_h(0)=0\), and \(w_h\) and \(z_h\) are single-valued at \(r_i\in \mathcal F^0_h\) by 
\eqref{eq:hdg-transmission-w} and \eqref{eq:hdg-transmission-z},  local formulae \eqref{eq:hdg-local-reconstruction-w} and \eqref{eq:hdg-local-reconstruction-z} patch together into the global identity
\begin{equation}\label{eq:w_z}
w_h(r)=\int_0^r u_h(s)\,ds, \qquad z_h(r)=\int_0^r g_h(s)\,ds.
\end{equation}
For \(g_h\), the continuity relation \eqref{eq:hdg-transmission-g} together with
\(\widehat g_h(b)=1\) shows that the elementwise formula
\eqref{eq:hdg-local-reconstruction-g} patches into the global representation
\[
g_h(r)=\exp\!\Bigl(-\int_r^b \frac1s\bigl(u_h(s)-\tilde u_h(s)\bigr)^2\,ds\Bigr).
\]
This proves \eqref{eq:disc-global-representations}. The bound \(0<g_h\le1\) follows
immediately from the exponential representation, and \eqref{eq:disc-g-strong} implies that
\(g_h\) is nondecreasing in \(r\). Since
\begin{equation}\label{tildeg_Aux}
\tilde g_h(r)= \frac{z_h(r)}{r} = \frac{1}{r}\int_0^r g_h(s)\,ds,
\end{equation}
we obtain \(0< \tilde g_h(r)\le g_h(r)\le1\) for \(r>0\). Finally, differentiating
\(z_h=r\tilde g_h\) and using \eqref{eq:disc-z-strong} yields
\[
(\tilde g_h)_r = \frac{1}{r}\,(g_h-\tilde g_h),
\]
and the limit \( \tilde u_h(0)=u_h(0)\) follows from \eqref{eq:w_z}, and the limit \(\tilde g_h(0)=g_h(0)\) follows from \eqref{tildeg_Aux}. This proves \eqref{eq:disc-metric-bounds}--\eqref{eq:disc-tildeg-identity}.
\end{proof}

\section{Stability}\label{sec:hdg-stability}

We now show that the local reconstruction preserves the exact metric identities of the continuous problem and therefore yields an unconditional \(L^2\)-stability estimate.

\begin{theorem}[\(L^2\)-stability]
\label{thm:hdg-L2-stability}
Let \((u_h,w_h,g_h,z_h,\widehat u_h,\widehat w_h,\widehat g_h,\widehat z_h)\) be a semidiscrete
solution of \eqref{eq:hdg-local-reconstruction}--\eqref{eq:hdg-boundary-traces}.
Then, for all \(t\in[0,T]\),
\begin{equation}\label{eq:hdg-energy}
\frac12\frac{d}{dt}\|u_h(t)\|_{\mathcal T_h}^2 + \mathcal D_h(t)
\le \frac14\,U_b(t)^2 + \frac14\,\bigl(1-g_h(t,0)\bigr),
\end{equation}
where
\begin{align}
\mathcal D_h(t) := {}&
\frac14\,g_h(t,0)\,u_h(t,0)^2
+ \frac14\,\tilde g_h(t,b)\,\bigl(u_h(t,b)-U_b(t)\bigr)^2 \nonumber\\
&\qquad
+ \frac14\sum_{i=1}^{N-1}\tilde g_h(t,r_i)\,
\bigl(u_h|_{I_i}(r_i)(t)-\widehat u_h(r_i)(t)\bigr)^2 \ge 0.
\label{eq:hdg-dissipation}
\end{align}
\end{theorem}

\begin{proof}
Take \(\phi_u=u_h\) in \eqref{eq:hdg-local-u} on each element \(I_i\), an elementwise integration by parts and summing over all elements yields
\begin{equation}\label{eq:energy-global}
\frac12\frac{d}{dt}\|u_h\|_{\mathcal T_h}^2
+\Bigl\langle \frac14\,\tilde g_h\,u_h^2
-\widehat{\mathcal F}_{u,h}\,u_h,\;n\Bigr\rangle_{\partial\mathcal T_h}
=
\frac14\,\bigl((\tilde g_h)_r,\;u_h^2-2u_h\tilde u_h\bigr)_{\mathcal T_h}.
\end{equation}
By Lemma~\ref{lem:hdg-metric-reconstruction}, \(
(\tilde g_h)_r\ge0,~
(\tilde g_h)_r\le \frac{g_h}{r}.
\)
Moreover, \( u_h^2-2u_h\tilde u_h \le (u_h-\tilde u_h)^2\). Therefore,
\begin{align}
\frac14\,\bigl((\tilde g_h)_r,\;u_h^2-2u_h\tilde u_h\bigr)_{\mathcal T_h}
&\le
\frac14\left(\frac{g_h}{r}(u_h-\tilde u_h)^2,1\right)_{\mathcal T_h} \nonumber\\
&=
\frac14\,\bigl((g_h)_r,1\bigr)_{\mathcal T_h}
=
\frac14\int_0^b (g_h)_r\,dr \nonumber\\
&=
\frac14\,\bigl(g_h(b)-g_h(0)\bigr)
=
\frac14\,\bigl(1-g_h(0)\bigr),
\label{eq:volume-bound}
\end{align}
where we used \eqref{eq:disc-g-strong} and \(\widehat g_h(b)=1\).
It remains to analyze the boundary pairing in \eqref{eq:energy-global}. For an
interior node \(r_i\in\mathcal F_h^0\), define
\[
\Phi_i :=
\Bigl(\frac14\,\tilde g_h\,u_h^2-\widehat{\mathcal F}_{u,h}\,u_h\Bigr)\Big|_{I_i}(r_i)
-
\Bigl(\frac14\,\tilde g_h\,u_h^2-\widehat{\mathcal F}_{u,h}\,u_h\Bigr)\Big|_{I_{i+1}}(r_i).
\]
By the definition of \(\tilde g\) and \eqref{eq:hdg-transmission-z}, \(\tilde g_h(r_i)\) is single-valued. Set
\( F_i := \widehat{\mathcal F}_{u,h}|_{I_i}(r_i) = \widehat{\mathcal F}_{u,h}|_{I_{i+1}}(r_i)
\), where the equality comes from \eqref{eq:hdg-transmission-u}. Thus, at \(r_i\), we take
\[
F_i = \frac12\,\tilde g_h(r_i)\,\widehat u_h(r_i).
\]
Hence
\begin{align}
\Phi_i
&=
\frac14\,\tilde g_h(r_i)\,(u_h|_{I_i}(r_i))^2 - F_i\,u_h|_{I_i}(r_i)
-\frac14\,\tilde g_h(r_i)\,(\widehat u_h(r_i))^2 + F_i\,\widehat u_h(r_i) \nonumber\\
&=
\frac14\,\tilde g_h(r_i)\,(u_h|_{I_i}(r_i))^2
-\frac12\,\tilde g_h(r_i)\,\widehat u_h(r_i)u_h|_{I_i}(r_i)
+\frac14\,\tilde g_h(r_i)\,(\widehat u_h(r_i))^2 \nonumber\\
&=
\frac14\,\tilde g_h(r_i)\,\bigl(u_h|_{I_i}(r_i)-\widehat u_h(r_i)\bigr)^2.
\label{eq:interior-dissipation}
\end{align}
At the left boundary \(r=0\), using \(\tilde g_h(0)=g_h(0)\) and the left end flux
\(
\widehat{\mathcal F}_{u,h}(0)=\frac12\,g_h(0)\,u_h(0),
\) we obtain
\begin{equation}\label{eq:left-boundary-dissipation}
\Phi_0
=
-\Bigl(\frac14\,g_h(0)\,u_h(0)^2-\frac12\,g_h(0)\,u_h(0)^2\Bigr)
=
\frac14\,g_h(0)\,u_h(0)^2.
\end{equation}
At the right boundary \(r=b\), the numerical flux is \(
\widehat{\mathcal F}_{u,h}(b)=\frac12\,\tilde g_h(b)\,U_b\),
and therefore
\begin{align}
\Phi_b
&=
\frac14\,\tilde g_h(b)\,u_h(b)^2
-\frac12\,\tilde g_h(b)\,U_b\,u_h(b) =
\frac14\,\tilde g_h(b)\,\bigl(u_h(b)-U_b\bigr)^2
-\frac14\,\tilde g_h(b)\,U_b^2 \nonumber\\
&\ge
\frac14\,\tilde g_h(b)\,\bigl(u_h(b)-U_b\bigr)^2
-\frac14\,U_b^2,
\label{eq:right-boundary-dissipation}
\end{align}
since \(0\le \tilde g_h(b)\le1\). Collecting \eqref{eq:interior-dissipation},
\eqref{eq:left-boundary-dissipation}, and
\eqref{eq:right-boundary-dissipation}, we end up with
\begin{equation}\label{eq:boundary-pairing}
\Bigl\langle \frac14\,\tilde g_h\,u_h^2-\widehat{\mathcal F}_{u,h}\,u_h,\;n\Bigr\rangle_{\partial\mathcal T_h}
\ge
\mathcal D_h(t)-\frac14\,U_b(t)^2.
\end{equation}
Substituting \eqref{eq:volume-bound} and
\eqref{eq:boundary-pairing} into \eqref{eq:energy-global} yields
\eqref{eq:hdg-energy}.
\end{proof}

\begin{corollary}[Global well-posedness]
\label{cor:global-wellposedness-hdg}
Under the assumptions of Theorem~\ref{thm:local-wellposedness-hdg}, the local
solution extends uniquely to all of \([0,T]\).
\end{corollary}

\begin{proof}
By Theorem~\ref{thm:hdg-L2-stability}, \(\|u_h(t)\|_{\mathcal T_h}\) remains bounded on every bounded time interval. Since \(V_h^k\) is finite-dimensional, the coefficient vector of \(u_h\) remains bounded as well. The standard continuation criterion for finite-dimensional ODEs then yields  a unique extension to the whole
interval \([0,T]\).
\end{proof}

%%%%%%%%%% Error Analysis; for k \ge 1 %%%%%%%% 

\section{Optimal error analysis}\label{sec:hdg-error}

We assume that the exact solution of \eqref{eq:ES_local} satisfies
\begin{equation}\label{eq:hdg-regularity}
% u\in W^{1,\infty}(0,T;W^{k+1,\infty}(0,b)).
u\in L^\infty(0,T;W^{k+1,\infty}(0,b))
\quad \text{ and } \quad
u_t\in L^\infty(0,T;H^{k+1}(0,b)).
\end{equation}
Then the exact quantities $w$, $g$, $z$, $\tilde u$, and $\tilde g$ are sufficiently smooth on $[0,T]\times[0,b]$, and in particular
$\tilde g_r\in L^{\infty}((0,T)\times(0,b))$. 
\begin{remark}\label{rem:exact-solution-input}
The regularity assumption imposed above is used only as an input for the numerical analysis on the fixed time interval $[0,T]$. In this setting, classical and generalized solution theories are available in the work
of Christodoulou \cite{christodoulou1986global} on $[0,\infty)\times(0,\infty)$, which includes large-data regimes. The recent DG analysis of Chen et al. \cite{chen2026convergence} derives a priori estimates and global existence on $[0,\infty)\times[0,b]$ for a class of large initial data, including $r=0$. 
\end{remark}

On each element $I_i$, we define $\Pi u\in P^k(I_i)$ by
\begin{equation}\label{eq:hdg-left-radau}
(\Pi u-u,v)_{I_i}=0 \qquad \forall v\in P^{k-1}(I_i),
\qquad
(\Pi u)(r_{i-1})=u(r_{i-1}).
\end{equation}
This is the Gauss--Radau projection; see, for example,
\cite{cockburn1989tvbii,richter1988optimal}. We set
\begin{equation}\label{eq:hdg-error-splitting}
\eta:=u-\Pi u,
\qquad
\xi:=u_h-\Pi u,
\qquad
e:=u-u_h=\eta-\xi,
\end{equation}
and
\begin{equation}\label{eq:hdg-error-metric-def}
\tilde e:=\tilde u-\tilde u_h,
\qquad
\delta_g:=g-g_h,
\qquad
\widetilde\delta_g:=\tilde g-\tilde g_h.
\end{equation}
We choose the initial data by \( u_h(0)=\Pi u(0),
\) so that $\xi(0)=0$. The standard approximation, inverse, and trace estimates yield
\begin{subequations}\label{eq:hdg-standard-estimates}
\begin{align}
\|\eta\|_{\Th}
+h\|\eta_r\|_{\Th}
+\|\eta\|_{L^{\infty}(0,b)}
+h\|\eta_r\|_{L^{\infty}(0,b)}
+\|\eta_t\|_{\Th}
&\le Ch^{k+1},
\label{eq:hdg-standard-estimates-a}\\
\left(\sum_{i=1}^N \big|\eta|_{I_i}(r_i)\big|^2\right)^{1/2}
&\le Ch^{k+\frac12},
\label{eq:hdg-standard-estimates-b}\\
\|v_h\|_{L^{\infty}(0,b)}
&\le Ch^{-\frac12}\|v_h\|_{\Th},
\label{eq:hdg-standard-estimates-c}\\
\|(v_h)_r\|_{\Th}
&\le Ch^{-1}\|v_h\|_{\Th},
\label{eq:hdg-standard-estimates-d}\\
\left(\sum_{i=1}^N \big|v_h|_{I_i}(r_{i-1})\big|^2
+\sum_{i=1}^N \big|v_h|_{I_i}(r_i)\big|^2\right)^{1/2}
&\le Ch^{-\frac12}\|v_h\|_{\Th}
\label{eq:hdg-standard-estimates-e}
\end{align}
\end{subequations}
 for all $v_h\in V_h^k$, where $\eta_r|_{I_i} = \big(u_r - \Pi u\big)_r|_{I_i}$ and $\eta_t = u_t - \Pi u_t= \big(u - \Pi u\big)_t$; see \cite[Ch. 3]{ciarlet1978finite}. 
Now we make the following bootstrap assumption:
\begin{equation}\label{eq:hdg-bootstrap}
\|\xi(t)\|_{\Th}\le h^{\frac32},
\qquad 0\le t\le T.
\end{equation}
By \eqref{eq:hdg-error-splitting}, \eqref{eq:hdg-standard-estimates}, and \eqref{eq:hdg-bootstrap}, we obtain
\begin{equation}\label{eq:hdg-bootstrap-consequence}
\|e(t)\|_{L^{\infty}(0,b)} \le \| \xi(t)\|_{L^{\infty}(0,b)} + \|\eta(t)\|_{L^{\infty}(0,b)}  \le Ch,
\qquad
\|e_r(t)\|_{\Th}\le Ch^{\frac12},
\qquad 0\le t\le T.
\end{equation}

\begin{lemma}\label{lem:hdg-auxiliary-estimates}
Assume \eqref{eq:hdg-bootstrap}. Then, for $0\le t\le T$,
\begin{subequations}\label{eq:hdg-auxiliary-estimates}
\begin{align}
\|\tilde e\|_{L^{\infty}(0,b)}
&\le C\bigl(h^{k+1}+h^{-\frac12}\|\xi\|_{\Th}\bigr),
\label{eq:hdg-auxiliary-estimates-a}\\
\|(\tilde e)_r\|_{L^{\infty}(0,b)}
&\le C\bigl(h^k+h^{-\frac32}\|\xi\|_{\Th}\bigr),
\label{eq:hdg-auxiliary-estimates-b}\\
\|(\tilde e)_r\|_{\Th}
&\le C\bigl(h^k+h^{-1}\|\xi\|_{\Th}\bigr),
\label{eq:hdg-auxiliary-estimates-c}\\
\|\delta_g\|_{L^{\infty}(0,b)}+\|\widetilde\delta_g\|_{L^{\infty}(0,b)}
&\le C\|e\|_{\Th},
\label{eq:hdg-auxiliary-estimates-d}\\
\|(\delta_g)_r\|_{\Th}+\|(\widetilde\delta_g)_r\|_{\Th}
&\le C\|e\|_{\Th},
\label{eq:hdg-auxiliary-estimates-e}
\end{align}
\end{subequations}
where the constant $C>0$ is generic and independent of $h$.
\end{lemma}

\begin{proof}
For any $v\in H^1(0,b)$, we write
\[
\widetilde v(r):=\frac1r\int_0^r v(s)\,ds = \int_0^1 v(\theta r)\,d\theta,
\qquad
(\widetilde v)_r(r)=\int_0^1 \theta\,v_r(\theta r)\,d\theta.
\]
Thus, by H\"older's inequality, it is easy to see that
\begin{equation}\label{eq:average_bound}
\|\widetilde v\|_{L^{\infty}(0,b)}\le \|v\|_{L^{\infty}(0,b)},
~~
\|\widetilde v\|_{\Th}\le C\|v\|_{\Th},
~~
\|(\widetilde v)_r\|_{L^{\infty}(0,b)}\le \|v_r\|_{L^{\infty}(0,b)},
~~
\|(\widetilde v)_r\|_{\Th}\le C\|v_r\|_{\Th}.
\end{equation}
Applying these bounds to $v=e=\eta-\xi$ and using \eqref{eq:hdg-standard-estimates} gives
\eqref{eq:hdg-auxiliary-estimates-a}--\eqref{eq:hdg-auxiliary-estimates-c}. Next, set
\[
\alpha:=u-\tilde u,
\qquad
\alpha_h:=u_h-\tilde u_h.
\]
Then, using definition \eqref{eq:ES_tildeu} and $\tilde u_h(r)=\frac{w_h(r)}{r}= \frac{1}{r}\int_0^r u_h(s)\,ds$ from \eqref{eq:disc-global-representations}, for $r\in I_i$, we have
\[
\frac{\alpha}{r}= \frac{1}{r}\big((r\tilde u)_r -\tilde u\big) =  (\tilde u)_r,
\qquad
\frac{\alpha_h}{r}=\frac{1}{r} \big((r\tilde u_h)_r- \tilde u_h\big )= (\tilde u_h)_r=(\tilde u)_r-(\tilde e)_r, 
\]
for all $i =1,\dots,N$. By \eqref{eq:hdg-auxiliary-estimates-b} and \eqref{eq:hdg-bootstrap}, we write
\[
\left\|\frac{\alpha_h}{r}\right\|_{L^\infty(0,b)}
\le \|(\tilde u)_r\|_{L^\infty(0,b)}+\|(\tilde e)_r\|_{L^\infty(0,b)}
\le C + C\bigl(h^k+h^{-3/2}\|\xi\|_{\Th}\bigr)
\le C,
\]
for $h\leq 1$, and therefore 
\begin{equation}\label{eq:hdg-alpha-bound}
\left\|\frac{|\alpha|+|\alpha_h|}{r}\right\|_{L^{\infty}(0,b)}\le C.
\end{equation}
By \eqref{ES_gatrb} and \eqref{eq:disc-global-representations},
\[
g(r)=\exp\!\Bigl(-\int_r^b \frac1s\,\alpha(s)^2\,ds\Bigr),
\qquad
g_h(r)=\exp\!\Bigl(-\int_r^b \frac1s\,\alpha_h(s)^2\,ds\Bigr).
\]
Since the exponential map is Lipschitz on $[0,\infty)$, we obtain
\begin{align}\label{eq:g-g_h}
\nonumber|\delta_g(r)|
&\le \int_r^b \frac1s\,\bigl|\alpha(s)^2-\alpha_h(s)^2\bigr|\,ds\\ \nonumber
&\le \left\|\frac{|\alpha|+|\alpha_h|}{r}\right\|_{L^{\infty}(0,b)}
\int_r^b |\alpha(s)-\alpha_h(s)|\,ds\\
&\le C\|\alpha-\alpha_h\|_{\Th}.
\end{align}
Now \(
\alpha-\alpha_h=e-\tilde e,
\)
so the bounds \eqref{eq:average_bound} imply
\(
\|\alpha-\alpha_h\|_{\Th}\le  \|e\|_{\Th}+\|\tilde e\|_{\Th}
\le C\|e\|_{\Th}.
\)
This proves \eqref{eq:hdg-auxiliary-estimates-d} for $\delta_g$, and the estimate for
$\widetilde\delta_g$ follows from
\[
\widetilde\delta_g(r)=\frac1r\int_0^r \delta_g(s)\,ds.
\]
For the derivatives, using \eqref{eq:ES_g} and \eqref{eq:disc-g-strong}, we obtain
\[
(\delta_g)_r
=\frac1r\bigl(g\alpha^2-g_h\alpha_h^2\bigr)
=\delta_g\,\frac{\alpha^2}{r}+g_h\,\frac{\alpha+\alpha_h}{r}\,(\alpha-\alpha_h).
\]
Hence, by \eqref{eq:hdg-alpha-bound}, \eqref{eq:disc-metric-bounds}, and \eqref{eq:hdg-auxiliary-estimates-d},
\[
\|(\delta_g)_r\|_{\Th}\le C\|\delta_g\|_{L^{\infty}(0,b)}+C\|\alpha-\alpha_h\|_{\Th}
\le C\|e\|_{\Th}.
\]
The estimate for $(\widetilde\delta_g)_r$ again follows from the averaging operator formula \eqref{eq:average_bound}.
Thus, \eqref{eq:hdg-auxiliary-estimates-e} holds.
\end{proof}

The exact solution satisfies the corresponding identity and scheme \eqref{eq:hdg-local-u} with $u$ and $\tilde g$.
Subtracting the discrete relation from the exact one, taking $\phi_u=\xi$ in \eqref{eq:hdg-local-u},
and using \eqref{eq:hdg-left-radau} and \eqref{eq:hdg-error-splitting}, we obtain the following error equation:
\begin{align}
&\sum_{i=1}^N ({\eta}_t,\xi)_{I_i}
+\frac12\sum_{i=1}^N (\tilde g\,\eta,\xi_r)_{I_i}
+Q_1+Q_2 \nonumber\\& \qquad=
\sum_{i=1}^N \Bigl( ( \xi_t,\xi)_{I_i}
+\frac12(\tilde g\,\xi,\xi_r)_{I_i}\Bigr)
-\frac12\sum_{i=1}^{N-1}\tilde g(r_i)\,\xi|_{I_{i+1}}(r_i)\,\xi|_{I_i}(r_i)
+\frac12\sum_{i=1}^N \tilde g(r_{i-1})\,\xi|_{I_i}(r_{i-1})^2,
\label{eq:hdg-error-equation}
\end{align}
where
\begin{align}
Q_1:={}&-
\frac12\sum_{i=1}^{N}\bigl((u_h\widetilde\delta_g)_r,\xi\bigr)_{I_i}
+\frac12\sum_{i=1}^{N-1}\widetilde\delta_g(r_i)
\Bigl(u_h|_{I_i}(r_i)-u_h|_{I_{i+1}}(r_i)\Bigr)\xi|_{I_i}(r_i) \nonumber\\
&\qquad
+\frac12\widetilde\delta_g(b)\Bigl(u_h(b)-U_b\Bigr)\xi|_{I_N}(b),
\label{eq:hdg-Q2}\\
Q_2:={}&\frac12\sum_{i=1}^{N}\bigl(\tilde g_r\,\tilde e+(\widetilde\delta_g)_r\,\tilde u_h,\xi\bigr)_{I_i}.
\label{eq:hdg-Q3}
\end{align}
Indeed, the terms containing $\eta$ on the element boundaries vanish because
\[
\eta|_{I_i}(r_{i-1})=0,
\qquad
\eta|_{I_{i+1}}(r_i)=0,
\qquad i=1,\dots,N-1,
\]
by the definition of the projection \eqref{eq:hdg-left-radau}, and at $r=b$ the exact and discrete inflow data are both equal to $U_b$.

\begin{lemma}\label{lem:hdg-Q2Q3}
Assume \eqref{eq:hdg-bootstrap}. Then, for $0\le t\le T$,
\begin{equation}\label{eq:hdg-Q2Q3-bound}
|Q_1|+|Q_2|
\le C\bigl(\|\xi\|_{\Th}^2+\|\eta\|_{\Th}^2\bigr).
\end{equation}
\end{lemma}

\begin{proof}
We first estimate $Q_1$.
From \eqref{eq:hdg-Q2},
\begin{align}\label{eq:Q_1_integ}
\nonumber |Q_1|
\le{}&
\frac12\|(u_{h})_r\|_{\Th}\,\|\widetilde\delta_g\|_{L^{\infty}(0,b)}\,\|\xi\|_{\Th}
+\frac12\|u_h\|_{L^{\infty}(0,b)}\,\|(\widetilde\delta_g)_r\|_{\Th}\,\|\xi\|_{\Th}\\ \nonumber
&
+\frac12\|\widetilde\delta_g\|_{L^{\infty}(0,b)}
\sum_{i=1}^{N-1}
\Bigl|u_h|_{I_i}(r_i)-u_h|_{I_{i+1}}(r_i)\Bigr|\,\Bigl|\xi|_{I_i}(r_i)\Bigr|\\
&
+\frac12\|\widetilde\delta_g\|_{L^{\infty}(0,b)}\,|u_h(b)-U_b|\,|\xi|_{I_N}(b)|.
\end{align}
By \eqref{eq:hdg-bootstrap-consequence} and the regularity of the exact solution, \(
\|u_h\|_{L^{\infty}(0,b)}+\|(u_{h})_r\|_{\Th}\le C.
\)
Hence, this estimate and Lemma~\ref{lem:hdg-auxiliary-estimates} imply that first two volume terms of \eqref{eq:Q_1_integ} are bounded by
\[
\frac12\|(u_{h})_r\|_{\Th}\,\|\widetilde\delta_g\|_{L^{\infty}(0,b)}\,\|\xi\|_{\Th}
+\frac12\|u_h\|_{L^{\infty}(0,b)}\,\|(\widetilde\delta_g)_r\|_{\Th}\,\|\xi\|_{\Th}\le C\|e\|_{\Th}\,\|\xi\|_{\Th}
\le C\bigl(\|\xi\|_{\Th}^2+\|\eta\|_{\Th}^2\bigr).
\]
For the boundary terms in \eqref{eq:Q_1_integ}, since $(\Pi u)|_{I_{i+1}}(r_i)=u(r_i) $, we have
\[
u_h|_{I_i}(r_i)-u_h|_{I_{i+1}}(r_i)
=\xi|_{I_i}(r_i)-\xi|_{I_{i+1}}(r_i)-\eta|_{I_i}(r_i),
\qquad i=1,\dots,N-1,
\]
and at $r=b$,
\[
u_h(b)-U_b=\xi|_{I_N}(b)-\eta|_{I_N}(b).
\]
Therefore,
\begin{align*}
&\sum_{i=1}^{N-1}
\Bigl|u_h|_{I_i}(r_i)-u_h|_{I_{i+1}}(r_i)\Bigr|\,\Bigl|\xi|_{I_i}(r_i)\Bigr|
+|u_h(b)-U_b|\,\big|\xi|_{I_N}(b)\big|\\
&\qquad\le
C\Biggl(
\sum_{i=1}^{N}\big|\xi|_{I_i}(r_i)\big|^2
+\sum_{i=2}^{N}\big|\xi|_{I_i}(r_{i-1})\big|^2
+\sum_{i=1}^{N}\big|\eta|_{I_i}(r_i)\big|^2
\Biggr)\\
&\qquad\le Ch^{-1}\bigl(\|\xi\|_{\Th}^2+\|\eta\|_{\Th}^2\bigr)
\end{align*}
by \eqref{eq:hdg-standard-estimates-b} and \eqref{eq:hdg-standard-estimates-e}.
Since Lemma~\ref{lem:hdg-auxiliary-estimates} gives
\(
\|\widetilde\delta_g\|_{L^{\infty}(0,b)}\le C\|e\|_{\Th}
\), and thus using \eqref{eq:hdg-standard-estimates-a} and  \eqref{eq:hdg-bootstrap}, we get
\(
h^{-1}\|e\|_{\Th}\le Ch^k+Ch^{\frac12}\le C
\)
on $[0,T]$, the boundary part of $Q_1$ is also bounded by
$C(\|\xi\|_{\Th}^2+\|\eta\|_{\Th}^2)$. Thus
\[
|Q_1|\le C\bigl(\|\xi\|_{\Th}^2+\|\eta\|_{\Th}^2\bigr).
\]
For $Q_2$, by \eqref{eq:hdg-Q3},
\[
|Q_2|
\le C\|\tilde e\|_{\Th}\,\|\xi\|_{\Th}
+C\|(\widetilde\delta_g)_r\|_{\Th}\,\|\xi\|_{\Th},
\]
where we also used $\|\tilde u_h\|_{L^{\infty}(0,b)}\le C$.
By Lemma~\ref{lem:hdg-auxiliary-estimates} and \eqref{eq:average_bound},
\[
\|\tilde e\|_{\Th}\le C\|e\|_{\Th},
\qquad
\|(\widetilde\delta_g)_r\|_{\Th}\le C\|e\|_{\Th},
\]
so
\[
|Q_2|\le C\|e\|_{\Th}\,\|\xi\|_{\Th}
\le C\bigl(\|\xi\|_{\Th}^2+\|\eta\|_{\Th}^2\bigr).
\]
This proves \eqref{eq:hdg-Q2Q3-bound}.
\end{proof}

\begin{theorem}[Optimal error estimate]\label{thm:hdg-optimal-error}
Let $k\ge 1$.
Assume \eqref{eq:hdg-regularity} and choose the initial data by \( u_h(0)=\Pi u(0)
\).
Then there exist $h_0>0$ and $C>0$, independent of $h$, such that for all $0<h\le h_0$,
\begin{equation}\label{eq:hdg-optimal-error}
\sup_{0\le t\le T}\|u(t)-u_h(t)\|_{\Th}\le Ch^{k+1}.
\end{equation}
\end{theorem}

\begin{proof}
From \eqref{eq:hdg-error-equation},
\begin{align}
&\sum_{i=1}^N \Bigl((\xi_t,\xi)_{I_i}+\frac12(\tilde g\,\xi,\xi_r)_{I_i}\Bigr)
-\frac12\sum_{i=1}^{N-1}\tilde g(r_i)\,\xi|_{I_{i+1}}(r_i)\,\xi|_{I_i}(r_i)
+\frac12\sum_{i=1}^N \tilde g(r_{i-1})\,\xi|_{I_i}(r_{i-1})^2 \nonumber\\
&\qquad=
\sum_{i=1}^N (\eta_t,\xi)_{I_i}
+\frac12\sum_{i=1}^N (\tilde g\,\eta,\xi_r)_{I_i}
+Q_1+Q_2.
\label{eq:hdg-error-equation-2}
\end{align}
On the left-hand side, an elementwise integration by parts gives
\begin{align}
&\sum_{i=1}^N \Bigl((\xi_t,\xi)_{I_i}+\frac12(\tilde g\,\xi,\xi_r)_{I_i}\Bigr)
-\frac12\sum_{i=1}^{N-1}\tilde g(r_i)\,\xi|_{I_{i+1}}(r_i)\,\xi|_{I_i}(r_i)
+\frac12\sum_{i=1}^N \tilde g(r_{i-1})\,\xi|_{I_i}(r_{i-1})^2 \nonumber\\
&\qquad=
\frac12\frac{d}{dt}\|\xi\|_{\Th}^2
-\frac14(\tilde g_r,\xi^2)_{\Th}
+\frac14\tilde g(0)\,\xi|_{I_1}(0)^2
+\frac14\sum_{i=1}^{N-1}\tilde g(r_i)
\Bigl(\xi|_{I_i}(r_i)-\xi|_{I_{i+1}}(r_i)\Bigr)^2 \nonumber\\& \qquad\qquad
+\frac14\tilde g(b)\,\xi|_{I_N}(b)^2.
\label{eq:hdg-xi-energy}
\end{align}
Therefore,
\begin{align}\label{eq:hdg-xi-energy-lower}
\nonumber &\sum_{i=1}^N \Bigl((\xi_t,\xi)_{I_i}+\frac12(\tilde g\,\xi,\xi_r)_{I_i}\Bigr)
-\frac12\sum_{i=1}^{N-1}\tilde g(r_i)\,\xi|_{I_{i+1}}(r_i)\,\xi|_{I_i}(r_i)
+\frac12\sum_{i=1}^N \tilde g(r_{i-1})\,\xi|_{I_i}(r_{i-1})^2
\\& \qquad \ge
\frac12\frac{d}{dt}\|\xi\|_{\Th}^2-C\|\xi\|_{\Th}^2.
\end{align}
Here we used the boundedness of $\tilde g_r$ from \eqref{eq:hdg-regularity} and the nonnegativity of $\tilde g$ from Lemma \ref{lem:metric_bounds}. The first term on the right-hand side of \eqref{eq:hdg-error-equation-2} can be estimated using \eqref{eq:hdg-regularity} and \eqref{eq:hdg-standard-estimates} as follows
\[
\left|\sum_{i=1}^N (\eta_t,\xi)_{I_i}\right|
\le \|\eta_t\|_{\Th}\,\|\xi\|_{\Th}
\le Ch^{k+1}\|\xi\|_{\Th}
\le Ch^{2k+2}+C\|\xi\|_{\Th}^2.
\]
For the second term, let \(
\bar{\tilde g}_i:=\frac1{h_i}\int_{I_i}\tilde g(r)\,dr
 \in \mathbb R\). Since $\bar{\tilde g}_i\,\xi_r\in P^{k-1}(I_i)$, the orthogonality of the projection \eqref{eq:hdg-left-radau} gives \(
(\eta,\bar{\tilde g}_i\,\xi_r)_{I_i}=0
\). Hence
\[
(\tilde g\,\eta,\xi_r)_{I_i}=\bigl((\tilde g-\bar{\tilde g}_i)\eta,\xi_r\bigr)_{I_i}.
\]
Using $\|\tilde g-\bar{\tilde g}_i\|_{L^{\infty}(I_i)}\le Ch_i$ and
\eqref{eq:hdg-standard-estimates-d}, we obtain
\begin{align*}
\left|\frac12\sum_{i=1}^N (\tilde g\,\eta,\xi_r)_{I_i} \right| = \left|\frac12\sum_{i=1}^N \bigl((\tilde g-\bar{\tilde g}_i)\eta,\xi_r\bigr)_{I_i} \right|
\le Ch\|\eta\|_{\Th}\,\|\xi_r\|_{\Th}
\le Ch^{2k+2}+C\|\xi\|_{\Th}^2.
\end{align*}
Finally, Lemma~\ref{lem:hdg-Q2Q3} gives
\[
|Q_1|+|Q_2|\le C\bigl(\|\xi\|_{\Th}^2+\|\eta\|_{\Th}^2\bigr)
\le C\|\xi\|_{\Th}^2+Ch^{2k+2}.
\]
Combining these bounds with \eqref{eq:hdg-error-equation-2} and
\eqref{eq:hdg-xi-energy-lower}, we obtain
\[
\frac12\frac{d}{dt}\|\xi\|_{\Th}^2
\le C\|\xi\|_{\Th}^2+Ch^{2k+2},
\qquad 0\le t\le T.
\]
Since $\xi(0)=0$, Gr\"onwall's inequality yields
\begin{equation*}
\|\xi(t)\|_{\Th}\le Ch^{k+1},
\qquad 0\le t\le T.
\end{equation*}
Finally, the triangle inequality and the estimate \eqref{eq:hdg-standard-estimates} imply
\begin{equation*}
\|u(t)-u_h(t)\|_{\Th}\le Ch^{k+1},
\end{equation*}
for all $t\leq T$. This completes the proof.
\end{proof}
\medskip
\noindent
We now justify the bootstrap assumption \eqref{eq:hdg-bootstrap}.
By Corollary~\ref{cor:global-wellposedness-hdg}, we have
$u_h\in C^1([0,T];V_h^k)$, and hence the map
$t\mapsto \|\xi(t)\|_{\Th}$ is continuous.
The estimate proved above gives
\[
\sup_{0\le t\le T}\|\xi(t)\|_{\Th}\le Ch^{k+1}.
\]
Since $k\ge 1$, we have $k+1>\frac32$, and therefore
\[
Ch^{k+1}\le \frac12 h^{\frac32}
\]
for all sufficiently small $h$.
Thus, the estimate improves the bootstrap assumption.
The standard continuity argument then shows that \eqref{eq:hdg-bootstrap}
is valid on the whole interval $[0,T]$.

\begin{corollary}[Errors of the reconstructed variables]\label{cor:hdg-reconstruction-error}
Under the assumptions of Theorem~\ref{thm:hdg-optimal-error},
\begin{align}\label{eq:hdg-reconstruction-error}
\nonumber\sup_{0\le t\le T}
\Bigl(
\|w(t)-w_h(t)\|_{L^{\infty}(0,b)}
&+\|z(t)-z_h(t)\|_{L^{\infty}(0,b)}
+\|g(t)-g_h(t)\|_{L^{\infty}(0,b)}\\& \qquad
+\|\tilde g(t)-\tilde g_h(t)\|_{L^{\infty}(0,b)}
+\|\tilde u(t)-\tilde u_h(t)\|_{\Th}
\Bigr)
\le Ch^{k+1}.
\end{align}
\end{corollary}

\begin{proof}
By \eqref{eq:disc-global-representations} and the exact reconstruction formulas,
\[
w(r)-w_h(r)=\int_0^r \bigl(u(s)-u_h(s)\bigr)\,ds,
\qquad
z(r)-z_h(r)=\int_0^r \bigl(g(s)-g_h(s)\bigr)\,ds.
\]
Hence
\[
\|w-w_h\|_{L^{\infty}(0,b)}\le C\|u-u_h\|_{\Th},
\qquad
\|z-z_h\|_{L^{\infty}(0,b)}\le C\|g-g_h\|_{L^{\infty}(0,b)}.
\]
Moreover, Lemma~\ref{lem:hdg-auxiliary-estimates} gives
\[
\|g-g_h\|_{L^{\infty}(0,b)}+\|\tilde g-\tilde g_h\|_{L^{\infty}(0,b)}
\le C\|u-u_h\|_{\Th}.
\]
The averaging operator bound \eqref{eq:average_bound} also yields
\[
\|\tilde u-\tilde u_h\|_{\Th}\le C\|u-u_h\|_{\Th}.
\]
Combining these estimates with Theorem~\ref{thm:hdg-optimal-error} proves \eqref{eq:hdg-reconstruction-error}.
\end{proof}

\begin{corollary}[Error estimate for the Bondi mass]\label{cor:hdg-bondi-mass}
Define the mass aspect by
\begin{equation}\label{eq:bondi_mass}
\mathcal M(t,r):=\frac{r}{2}\left(1-\frac{\tilde g(t,r)}{g(t,r)}\right),
\qquad
m_h(t,r):=\frac{r}{2}\left(1-\frac{\tilde g_h(t,r)}{g_h(t,r)}\right),
\end{equation}
and define the Bondi masses
\( M(t):=\mathcal M(t,b)\) and \( M_h(t):=m_h(t,b).\)
Under the assumptions of Theorem~\ref{thm:hdg-optimal-error}, there exists a constant
$C>0$, independent of $h$, such that
\[
\sup_{0\le t\le T}|M(t)-M_h(t)|\le Ch^{k+1}.
\]
\end{corollary}

\begin{proof}
Since the outer boundary condition gives
\( g(t,b)=1\) and \(g_h(t,b)=1 \), we have
\[
M(t)=\frac{b}{2}\bigl(1-\tilde g(t,b)\bigr),
\qquad
M_h(t)=\frac{b}{2}\bigl(1-\tilde g_h(t,b)\bigr).
\]
Therefore
\[
|M(t)-M_h(t)|
=
\frac{b}{2}\,|\tilde g(t,b)-\tilde g_h(t,b)|.
\]
The result now follows from Corollary~\ref{cor:hdg-reconstruction-error}.
\end{proof}

\section{Numerical experiments}\label{sec:numerical}
In this section, we present numerical experiments for the HDG method \eqref{eq:hdg-local-u}-\eqref{eq:hdg-boundary-traces}.
The aims are twofold. First, we verify the optimal order of convergence proved by Theorem \ref{thm:hdg-optimal-error} and Corollary \ref{cor:hdg-reconstruction-error}. Second, we show that the method captures the main geometric features of the Einstein--scalar system \eqref{eq:ES_local}, in particular, the collapse toward black-hole-type profiles. The first two tests below are adapted from the benchmark problems in \cite{chen2026convergence}.

All computations are performed on uniform meshes of the interval $[0,b]$.  Since the HDG scheme in this paper is semidiscrete, we discretize the resulting ODE system in time using the classical fourth-order Runge--Kutta method \cite{cockburn2001runge,dwivedi2024stability,sun2019strong}. In the convergence test, the time step is chosen sufficiently small, for example $\Delta t = 0.01$, so that the temporal error is negligible compared with the spatial error. 

A useful feature of the present HDG formulation is that the variables $w_h$, $g_h$, $z_h$, $\tilde u_h$, and $\tilde g_h$ are reconstructed from $u_h$ through the discrete constraint relations. Therefore, at each Runge--Kutta stage, only the coefficients of $u_h$ are advanced in time, while the remaining quantities are recovered by elementwise integration and recursion. In particular, the method does not require a separate global evolution for the metric variables.
This considerably reduces the implementation cost and is one of the main advantages of the present HDG formulation.

\subsection{Example 1}\textbf{(Accuracy test and large-data collapse benchmark; cf. \cite[Ex. 1]{chen2026convergence}).}
\label{subsec:numerical-example1}
We take
\[
b=10,
\qquad
\tilde u_0(r)=0.45\tanh(3(r-5)),
\]
and define the initial data for the evolution variable by
\(
u_0(r)=(r\tilde u_0(r))_r.
\)
The outer boundary data are chosen as
\(
U_b(t)=u_0(b).
\)
This is a large data set, and it leads to a rapid transition to the formation of a black-hole region \cite{chen2026convergence,christodoulou1987mathematical}.
The profile of $\tilde u_0$ is steep near $r=5$, which produces a strong inward concentration
of the scalar field during the evolution.

To test the spatial accuracy of the method, we compute the solution up to the final time $T=0.5$.
Since no explicit exact solution is available, we use a reference solution computed on a very
fine mesh, for example, with $N_{\mathrm{ref}}=6400$, together with a sufficiently accurate polynomial approximation and a very small time step, see \cite{chen2026convergence}.
We consider $P^k$ elements for $1\le k\le 5$.
The errors are measured by
\[
E_u
=
\|u_{\mathrm{ref}}(\cdot,T)-u_h(\cdot,T)\|_{L^2(0,b)},
\qquad
E_g
=
\|g_{\mathrm{ref}}(\cdot,T)-g_h(\cdot,T)\|_{L^2(0,b)},
\]
where $u_{\mathrm{ref}}$ and $g_{\mathrm{ref}}$ are reference solutions computed at $N_{\mathrm{ref}}$.  The results in Figure \ref{fig:hdg-example1-convergence} show the expected order \(k+1\) for \(u_h\). For the reconstructed metric variable \(g_h\), the observed slope is close to \(k+2\), which is better than the bound proved in Corollary \ref{cor:hdg-reconstruction-error}. At present, this higher rate should be
viewed as an observed superconvergent effect of the reconstruction rather than as a rigorously established theorem. A plausible explanation is that \(g_h\) is obtained through the nonlinear integral reconstruction \eqref{eq:hdg-local-reconstruction-g},
which may exhibit additional cancellations on uniform meshes in the smooth regime. We therefore interpret this behavior as an observed higher-order effect of the metric reconstruction.

To interpret the long-time behavior, we compare the numerical profiles with the
continuous asymptotic picture described in \cite{chen2026convergence,christodoulou1987mathematical}.
In particular, if the continuous solution has a nonzero final mass
\( M_\infty := \lim_{t\to\infty} M(t)\), then one expects convergence toward a black-hole end state with a horizon radius \(2M_\infty\). In the computations, we monitor the discrete finite-radius mass proxy \(M_h(t)\) defined in \eqref{eq:bondi_mass}, which serves as the numerical counterpart of this quantity.
In the normalization $g(t,b)=1$, this corresponds formally to the limiting profile; see \cite{christodoulou1987mathematical}:
\[
g(t,r) \to g_\infty(r)=
\begin{cases}
0, & 0\le r<2M_\infty,\\[0.3em]
1, & 2M_\infty\le r\le b,
\end{cases} \qquad \text{ as } t\to \infty.
\]
While $\tilde g$ is the radial average of $g$,
\[
\tilde g_\infty(r)=
\begin{cases}
0, & 0\le r\le 2M_\infty,\\[0.3em]
1-\dfrac{2M_\infty}{r}, & 2M_\infty<r\le b.
\end{cases}
\]
Hence $g$ is expected to approach a step profile, whereas $\tilde g$ remains continuous and develops a kink at $r=2M_\infty$ as $t \to \infty
$. This is consistent with the numerical results given in Figure \ref{fig:hdg-example1-profiles}, where we plot the profiles of $\tilde u_h$, $\tilde g_h$, and $g_h$ at several representative times.

\begin{figure}[htbp]
\centering
\includegraphics[width=0.82\textwidth]{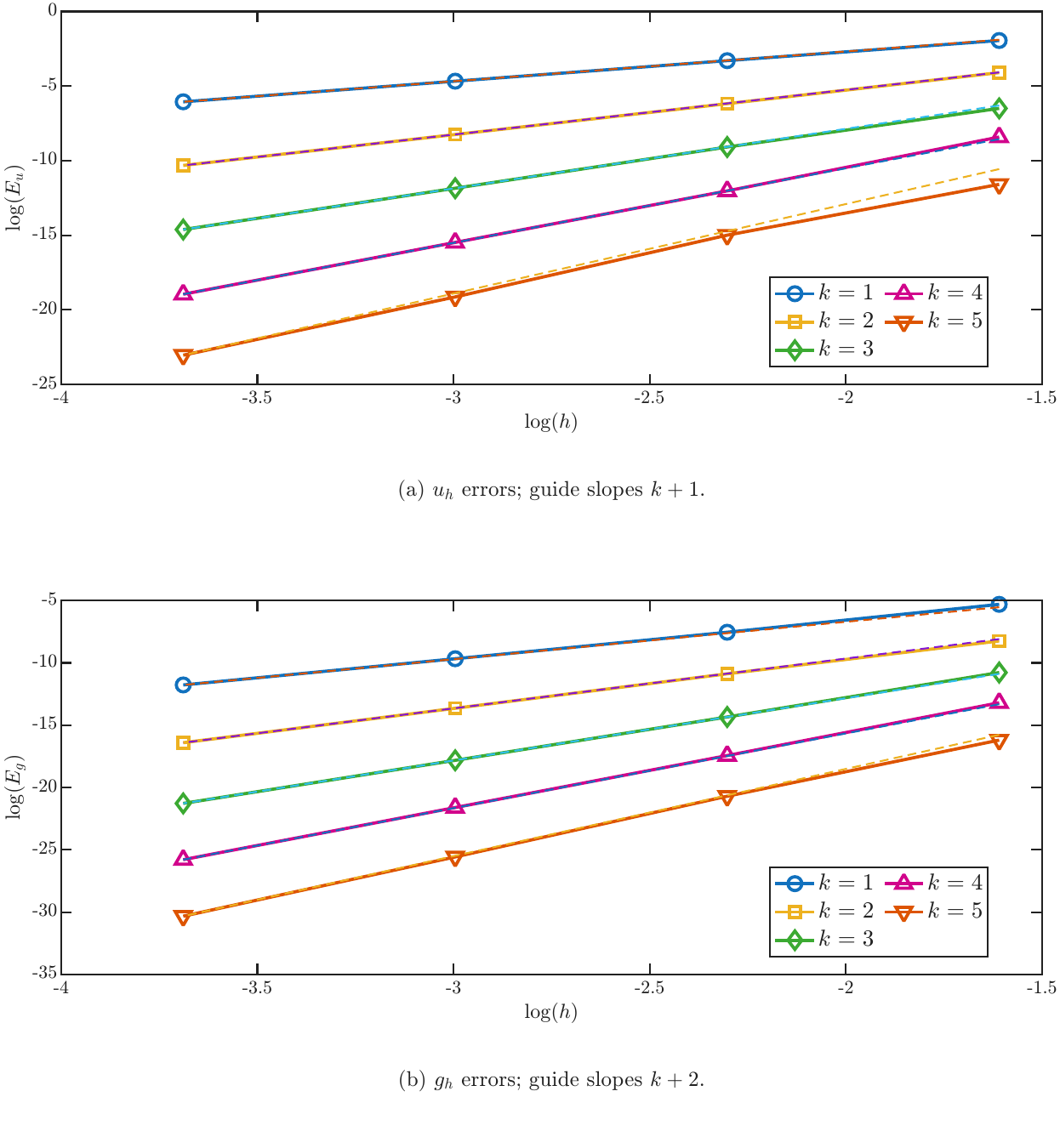}
\caption{Convergence history for Example~\ref{subsec:numerical-example1}. Panel (a) shows the error of $u_h$ with guide slopes $k+1$, while panel (b) shows the error of $g_h$ with guide slopes $k+2$.}
\label{fig:hdg-example1-convergence}
\end{figure}

\begin{figure}[htbp]
\centering
\includegraphics[width=0.92\textwidth]{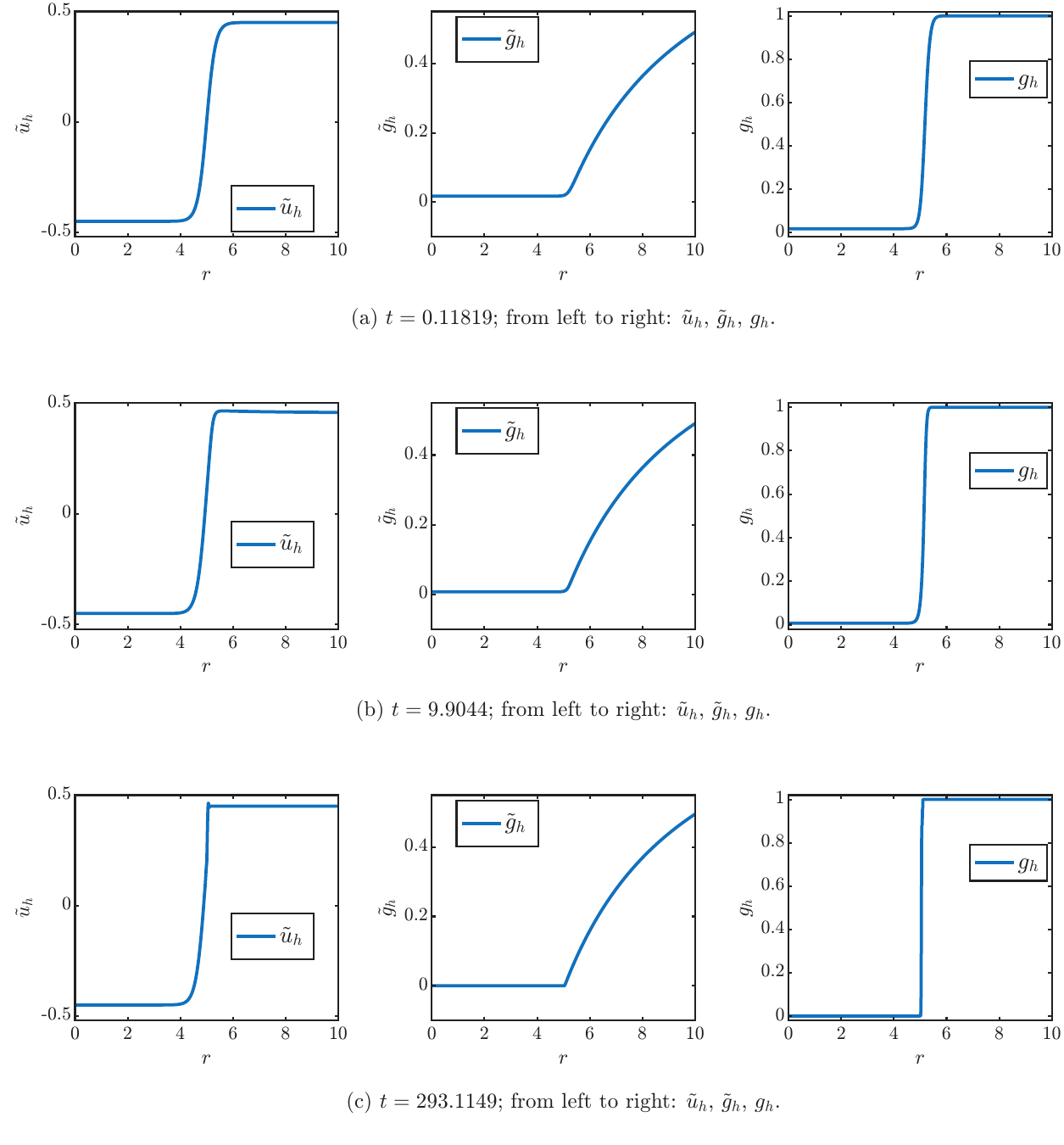}
\caption{Solutions of the Einstein--scalar system for Example~\ref{subsec:numerical-example1}. In each row, from left to right, the three panels show $\tilde u_h$, $\tilde g_h$, and $g_h$. }
\label{fig:hdg-example1-profiles}
\end{figure}

\subsection{Example 2} \textbf{(Long-time evolution for a lower-energy large data set \cite[Ex. 2]{chen2026convergence}).}
\label{subsec:numerical-example2} In this example, we consider a large data set with a smaller effective concentration, so that the collapse develops more slowly and the intermediate evolution can be observed more clearly.
We take
\[
b=10,
\qquad
\tilde u_0(r)=\tanh\bigl(m(r-r_c)\bigr),
\qquad
r_c=6,
\qquad
m=\frac{1}{5.1},
\]
and define
\[
u_0(r)=(r\tilde u_0(r))_r
=
\tanh\bigl(m(r-r_c)\bigr)
+
mr\,\operatorname{sech}^2\bigl(m(r-r_c)\bigr).
\]
The outer boundary data are again chosen as \(
U_b(t)=u_0(b).
\)
This initial condition has a finite Bondi mass and no horizon at the initial time. Compared with Example~\ref{subsec:numerical-example1}, the collapse is slower, so we present the numerical results of the
evolution over a sufficiently larger time. The numerical profiles of $\tilde u_h$, $\tilde g_h$, and $g_h$ are shown in Figure~\ref{fig:hdg-example2-profiles} at several representative times.

% -------------------- Example 5.2 --------------------
\begin{figure}[htbp]
\centering
\includegraphics[width=0.92\textwidth]{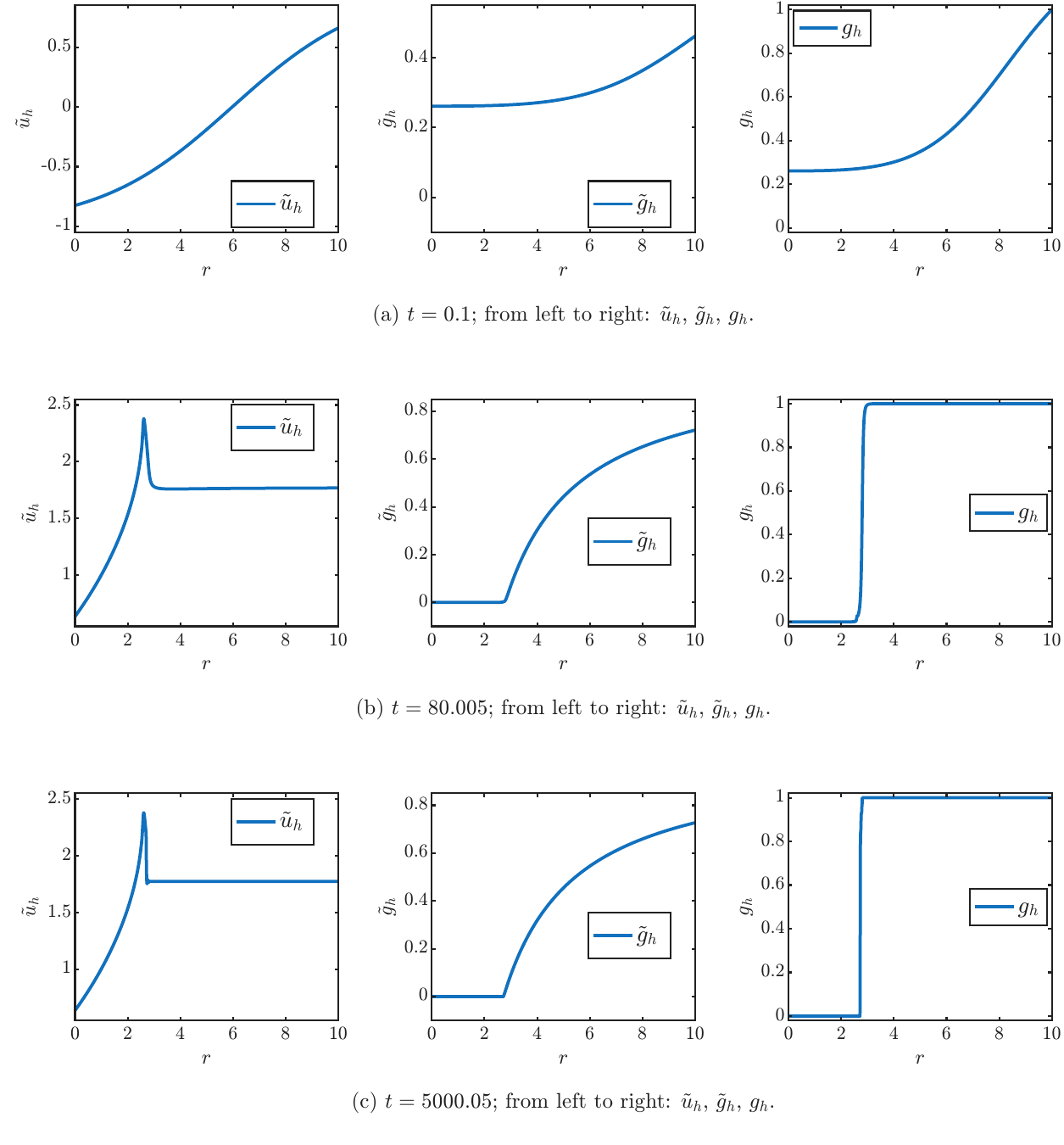}
\caption{Solutions of the Einstein--scalar system for Example~\ref{subsec:numerical-example2}. In each row, from left to right, the three panels show $\tilde u_h$, $\tilde g_h$, and $g_h$. }
\label{fig:hdg-example2-profiles}
\end{figure}

\subsection{Example 3} \textbf{(A Gaussian-pulse evolution).}
\label{subsec:numerical-example3}
In addition to the two large-data benchmark tests, it is useful to include a smoothly localized pulse that is not designed to form a black hole rapidly. A natural reference point is the Bondi-coordinate study of P\"urrer et al. \cite[Sec. 4]{purrer2005news}, who considered Gaussian-like initial data in their analysis of critical and subcritical Einstein--scalar evolutions. Motivated by this literature, we consider the localized pulse on the finite interval $[0,b]$:
\[
\tilde u_0(r)=A\,r^2\exp\!\left(-\frac{(r-r_0)^2}{\sigma^2}\right),
\qquad b=20,
\qquad A=8\times 10^{-3},
\qquad r_0=8,
\qquad \sigma=1.5,
\]
and define the evolution variable by
\[
 u_0(r)=(r\tilde u_0(r))_r
 =A\exp\!\left(-\frac{(r-r_0)^2}{\sigma^2}\right)
 \left(3r^2-\frac{2r^3(r-r_0)}{\sigma^2}\right).
\]
At the outer boundary, we impose zero inflow \(U_b(t)=0\).
The factor $r^2$ guarantees regularity at the center and produces a smooth pulse which is initially concentrated away from $r=0$.
Compared with Examples~\ref{subsec:numerical-example1} and
\ref{subsec:numerical-example2}, this test is not intended as a collapse
benchmark, but as a qualitative smooth-pulse evolution study on a larger computational domain.

Typical output times are chosen as
\[
t=0,
\qquad t=20,
\qquad t=60,
\]
which represent the initial pulse, an intermediate stage during inward motion and focusing, and a later stage where the pulse has moved further toward the center while remaining fully visible in the computational window.
In addition to the profiles of $\tilde u_h$, $\tilde g_h$, and $g_h$, we monitor the numerical Bondi-mass proxy $M_h(t)$ defined by \eqref{eq:bondi_mass}. The numerical results are displayed in Figure \ref{fig:hdg-example3-profiles}.
At $t=0$, the scalar profile $\tilde u_h$ is a smooth positive pulse centered away from the origin. At later times, the pulse moves inward and becomes more concentrated, while a small trailing undershoot appears behind the main peak. In addition to the profiles of \(\tilde u_h\), \(\tilde g_h\), and \(g_h\), we plot in
Figure~\ref{fig:hdg-example3-mass} the finite-radius Bondi-mass proxy $M_h(t)$, which follows from \eqref{eq:bondi_mass} and the boundary condition \(g_h(t,b)=1\). Over the displayed time interval, this quantity changes slowly. This supports the interpretation of the experiment as a smooth weak-field pulse evolution rather than a
rapid transition toward a black-hole-type end state. We emphasize that this figure is intended as a qualitative diagnostic of the reconstructed metric behavior.

\begin{figure}[htbp]
\centering
\includegraphics[width=0.92\textwidth]{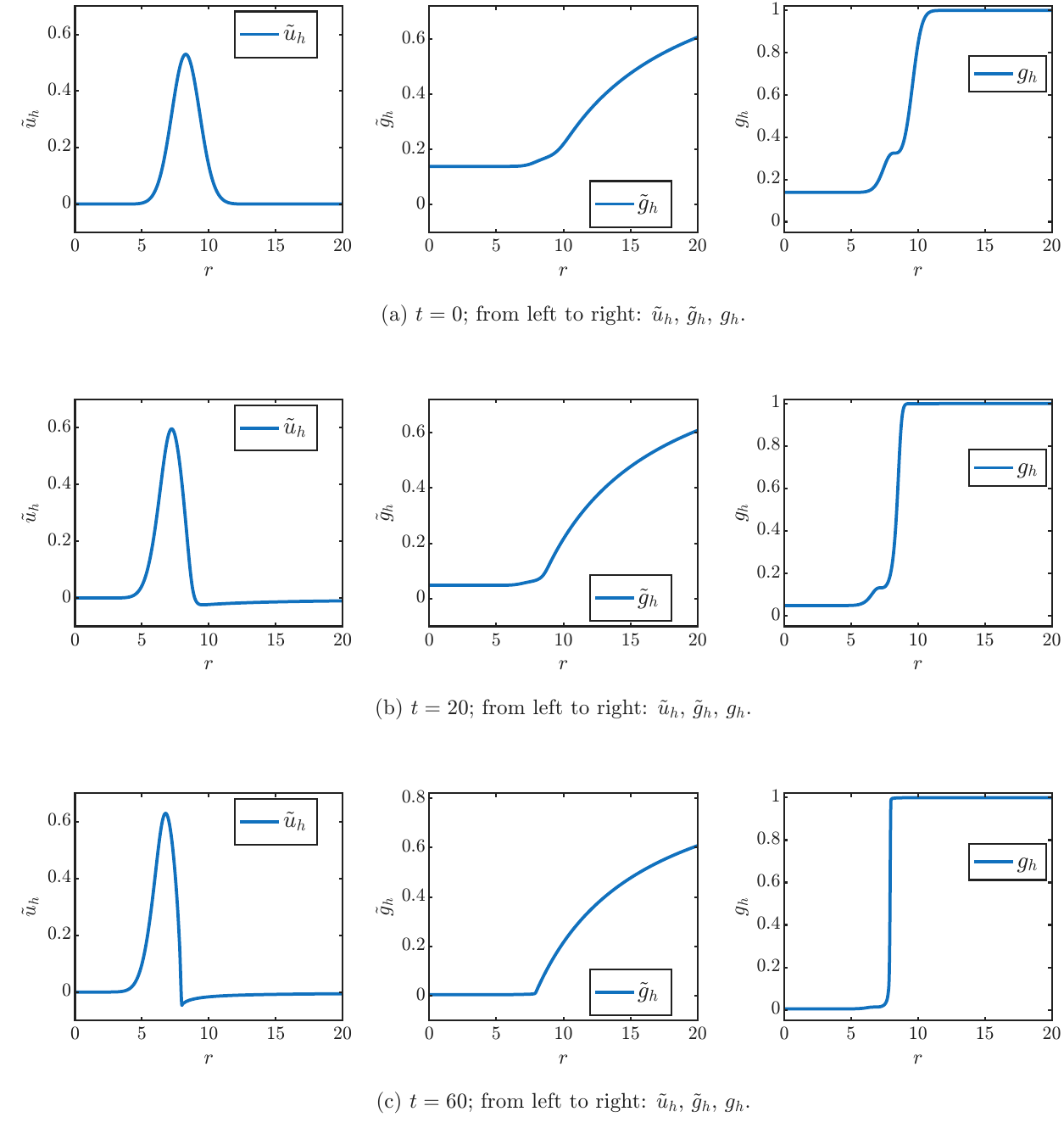}
\caption{Solutions of the Einstein--scalar system for the Gaussian-pulse test of Example~\ref{subsec:numerical-example3}. In each row, from left to right, the three panels show $\tilde u_h$, $\tilde g_h$, and $g_h$. }
\label{fig:hdg-example3-profiles}
\end{figure}

\begin{figure}[htbp]
\centering
\includegraphics[width=0.82\textwidth]{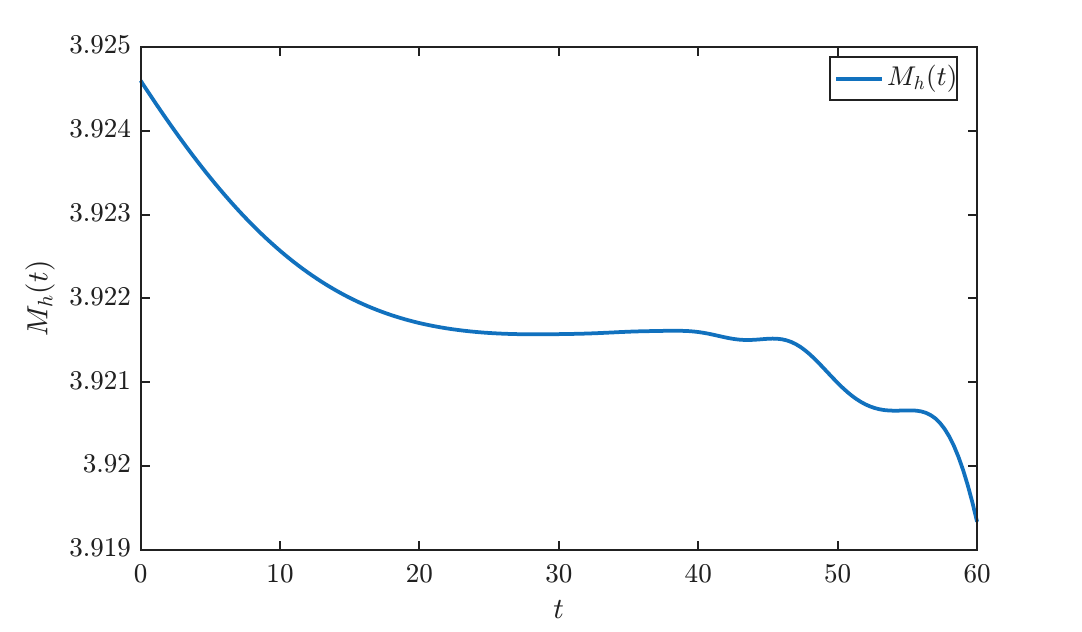}
\caption{Time history of the numerical Bondi-mass proxy for Example~\ref{subsec:numerical-example3}. }
\label{fig:hdg-example3-mass}
\end{figure}

\section{Conclusion}\label{sec:conclusion}
In this paper, we developed a hybridizable discontinuous Galerkin method for the spherically symmetric Einstein--scalar system in Bondi gauge. This method has a clear computational structure. The globally coupled unknowns are reduced to trace variables on the mesh skeleton, while the element variables are reconstructed locally.  We proved local well-posedness of the semidiscrete scheme and derived a global \(L^2\)-stability estimate. We then established an optimal order error bound for the main evolution variable for polynomial degree \(k\ge1\). The analysis further yielded error estimates for the reconstructed variables \(w_h\), \(g_h\), \(z_h\), \(\tilde u_h\), and \(\tilde g_h\), as well as an error estimate for the Bondi mass. These results show that the proposed HDG method is not only stable and convergent, but also accurate for the geometrically relevant quantities of the Einstein--scalar system. The numerical experiments support the theory and illustrate the practical behavior of the method. 

The present work provides the first HDG analysis for this Bondi-gauge Einstein--scalar model and shows that HDG offers a natural alternative to standard DG methods in this setting. Several extensions remain of interest. These include a fully discrete error analysis and a mathematical proof of the observed superconvergence of the reconstructed metric variable.

\section*{Code availability}
The code used to generate the numerical results reported in this paper is publicly available at \url{https://github.com/AndreasRupp/hdg_einstein_scalar}.

\bibliographystyle{abbrv}
\bibliography{Main_Eins}
\end{document}